\documentclass{amsart}
\usepackage{amsmath}
\usepackage{amsfonts}

\newtheorem{theorem}{Theorem}
\theoremstyle{plain}

\newtheorem{corollary}{Corollary}

\newtheorem{definition}{Definition}

\newtheorem{lemma}{Lemma}

\input{tcilatex}

\begin{document}
\title[Trace Formulas for a CFDO]{Trace Formulas for a Conformable
Fractional Diffusion Operator}

\begin{abstract}
In this paper, the regularized trace formulas for a diffusion operator which
include conformable fractional derivatives of order $\alpha $ $\left(
0<\alpha \leq 1\right) $ is obtained.
\end{abstract}

\author{YA\c{S}AR \c{C}AKMAK}
\maketitle

\section{\textbf{Introduction}}

The fractional derivative has an important place in applied mathematics.
Since 1695, the various types of fractional derivatives, usually given
through an integral form, have been introduced by many authors (see [1]-[4]).

In 2014, Khalil et al. gave the definition of conformable fractional
derivative\ (see [5]). Shortly after, Abdeljawad and Atangana et al. shown
the elementary properties of this derivative (see [6], [7]). The derivative
arises in various fields such as quantum mechanics, dynamical systems, time
scale problems, diffusions, conservation of mass, etc. (see [10]-[13]).

The trace of a matrix with finite dimensioanal is the sum of the elements on
the main diagonal and is finite. However, the trace of ordinary differential
operators with infinite dimensional, which is the sum of all eigenvalues, is
not finite. Therefore, the concept of regularize trace for this type
operators, which is finite, is mentioned. The regularized trace formulas
have great importance, especially, in the solution of the inverse problem
according to two spectra.

For about seventy years, the regularized trace formulas for the different
types of differential operators have been investigated. Firstly, in 1953,
Gelfand and Levitan obtained the regularized trace formula for the
Sturm-Liouville operator with Neumann conditions (see [14]). The study led
to the birth of a great and very important theory. After this study, the
theory has been continued for the various operators by many researchers (see
[15]-[46] and references therein). In 2010, Yang obtained the regularized
trace formulas for the diffusion operator i.e., a quadratic pencil of the
Schr\"{o}dinger operator (see [47]). In 2019, Mortazaasl and Jodayree
Akbarfam calculated the regularized trace formula for a conformable
fractional Sturm-Liouville problem (see [48]).

In the present paper, we consider a diffusion operator which include
conformable fractional derivatives of order $\alpha $ $\left( 0<\alpha \leq
1\right) $ instead of the ordinary derivatives in a traditional diffusion
operator. Using the contour integration method, we obtained the regularized
trace formulas for this diffusion operator.$\medskip $

\section{\textbf{Preliminaries}}

In this section, Firstly, we recall known some concepts of the conformable
fractional calculus. Then, we introduce a conformable fractional diffusion
operator with the separated boundary conditions on $\left[ 0,\pi \right] $.

\begin{definition}
Let $f:[0,\infty )\rightarrow 
\mathbb{R}
$ be a given function. Then, the conformable fractional derivative of $f$ of
order $\alpha $ with respect to $x$ is defined by%
\begin{equation*}
D_{x}^{\alpha }f(x)=\underset{h\rightarrow 0}{\lim }\dfrac{f(x+hx^{1-\alpha
})-f(x)}{h},\text{ }0<\alpha \leq 1,\text{ }x>0
\end{equation*}%
and%
\begin{equation*}
D_{x}^{\alpha }f(0)=\underset{x\rightarrow 0^{+}}{\lim }D_{x}^{\alpha }f(x).
\end{equation*}
If above limit exist and finite at any point $x$, $f$ is $\alpha -$%
differentiable at $x$ and 
\begin{equation*}
D_{x}^{\alpha }f(x)=x^{1-\alpha }f^{\prime }(x).
\end{equation*}
\end{definition}

\begin{definition}
Let $f:[a,\infty )\rightarrow 
\mathbb{R}
$ be a given function. The conformable fractional Integral of $f$ of order $%
\alpha $ is defined by%
\begin{equation*}
I_{\alpha }f(x)=\int_{0}^{x}f(t)d_{\alpha }t=\int_{0}^{x}t^{\alpha -1}f(t)dt,%
\text{ for all }x>0.
\end{equation*}
\end{definition}

\begin{lemma}
Let $f:[a,\infty )\rightarrow 
\mathbb{R}
$ be any continuous function. Then, 
\begin{equation*}
D_{x}^{\alpha }I_{\alpha }f(x)=f(x),\text{ for all }x>0.
\end{equation*}
\end{lemma}

\begin{lemma}
Let $f:(a,b)\rightarrow 
\mathbb{R}
$ be any differentiable function. Then, 
\begin{equation*}
I_{\alpha }D_{x}^{\alpha }f(x)=f(x)-f(a),\text{ for all }x>0.
\end{equation*}
\end{lemma}

\begin{definition}
\textbf{(}$\alpha -$\textbf{integration by parts): }Let $f,g:[a,b]%
\rightarrow 
\mathbb{R}
$ be two conformable fractional differentiable functions. Then,%
\begin{equation*}
\int_{a}^{b}f(x)D_{x}^{\alpha }g(x)d_{\alpha }x=\left. f(x)g(x)\right\vert
_{a}^{b}-\int_{a}^{b}g(x)D_{x}^{\alpha }f(x)d_{\alpha }x.
\end{equation*}
\end{definition}

\begin{lemma}
The space $L_{\alpha }^{2}\left( 0,a\right) $ associated with the inner
product for $f,$ $g\in L_{\alpha }^{2}\left( 0,a\right) $%
\begin{equation*}
\left\langle f,g\right\rangle :=\int_{0}^{a}f(x)\overline{g(x)}d_{\alpha }x
\end{equation*}%
is a Hilbert space.
\end{lemma}

\begin{definition}
The Sobolev space $W_{\alpha }^{2}\left( 0,a\right) $ consists of all
functions on $\left[ 0,a\right] $, such that $f(x)$ is absolutely continuous
and $D_{x}^{\alpha }f(x)\in L_{\alpha }^{2}\left( 0,a\right) .$
\end{definition}

\begin{definition}
Let $y\left( x\right) $ and $z\left( x\right) $ be $\alpha $-differentiable
functions on $\left[ 0,\pi \right] .$ The fractional Wronskian of $y\left(
x\right) $ and $z\left( x\right) $ is defined as 
\begin{equation*}
W_{\alpha }\left[ y\left( x\right) ,z\left( x\right) \right] :=\left\vert 
\begin{array}{cc}
y\left( x\right) & z\left( x\right) \medskip \\ 
D_{x}^{\alpha }y\left( x\right) & D_{x}^{\alpha }z\left( x\right)%
\end{array}%
\right\vert =y\left( x\right) D_{x}^{\alpha }z\left( x\right) -z\left(
x\right) D_{x}^{\alpha }y\left( x\right) .
\end{equation*}
\end{definition}

More detail knowledge about the conformable fractional calculus can be seen
in [5-9].

Now, let us consider the boundary value problem $L_{\alpha
}(p(x),q(x),h,H)=L_{\alpha }$, called as the conformable fractional
diffusion operator (CFDO), of the form%
\begin{eqnarray}
&&\text{\ }\left. \ell _{\alpha }y:=-D_{x}^{\alpha }D_{x}^{\alpha }y+\left[
2\lambda p(x)+q(x)\right] y=\lambda ^{2}y\text{, \ \ }0<x<\pi \right.
\medskip \\
&&\text{ }\left. U(y):=D_{x}^{\alpha }y(0)-hy(0)=0\right. \medskip \\
&&\text{ }\left. V(y):=D_{x}^{\alpha }y(\pi )+Hy(\pi )=0\right.
\end{eqnarray}%
where $\lambda $ is the spectral parameter$,$ $h,H\in 
\mathbb{R}
,$ $p(x),$ $D_{x}^{\alpha }p(x),$ $q(x)\in W_{\alpha }^{2}\left( 0,\pi
\right) $ are real valued functions, $0<\alpha \leq 1$ and $p(x)\neq $ const.

Let the functions $\varphi \left( x,\lambda \right) $ and $\psi \left(
x,\lambda \right) $ be the solutions of the equation $(1)$ satisfying the
initial conditions%
\begin{equation}
\varphi \left( 0,\lambda \right) =1,\text{ }D_{x}^{\alpha }\varphi \left(
0,\lambda \right) =h\text{,}
\end{equation}%
\begin{equation}
\psi \left( \pi ,\lambda \right) =1,\text{ }D_{x}^{\alpha }\psi \left( \pi
,\lambda \right) =-H
\end{equation}%
\newline
respectively. It is clear that $U\left( \varphi \right) =0,$ $V\left( \psi
\right) =0.$

Denote%
\begin{equation}
\Delta \left( \lambda \right) =W_{\alpha }\left[ \psi \left( x,\lambda
\right) ,\varphi \left( x,\lambda \right) \right] .
\end{equation}

It is proven in [49] that $W_{\alpha }$ does not depend on $x$ and putting $%
x=0$ and $x=\pi $ in (6) it can be written as%
\begin{equation}
\Delta \left( \lambda \right) =V\left( \varphi \right) =-U\left( \psi
\right) .
\end{equation}

\begin{definition}
The function $\Delta \left( \lambda \right) $ is called as the
characteristic function of the problem $L_{\alpha }.$

\begin{lemma}
For $\left\vert \lambda \right\vert \rightarrow \infty $ and each fixed $%
\alpha $ the following asymptotic formulas hold:%
\begin{equation}
\varphi \left( x,\lambda \right) =\cos \left( \tfrac{\lambda }{\alpha }%
x^{\alpha }-Q(x)\right) +O\left( \tfrac{1}{\left\vert \lambda \right\vert }%
\exp \left( \tfrac{\left\vert \tau \right\vert }{\alpha }x^{\alpha }\right)
\right) ,
\end{equation}%
\begin{equation}
D_{x}^{\alpha }\varphi \left( x,\lambda \right) =-\left( \lambda
-p(x)\right) \sin \left( \tfrac{\lambda }{\alpha }x^{\alpha }-Q(x)\right)
+O\left( \exp \left( \tfrac{\left\vert \tau \right\vert }{\alpha }x^{\alpha
}\right) \right)
\end{equation}%
\newline
where $\lambda =\sigma +i\tau $ and%
\begin{equation}
Q(x):=\int_{0}^{x}p(t)d_{\alpha }t.
\end{equation}

\begin{proof}
Firstly, we rewritten equation (1) as%
\begin{equation}
D_{x}^{\alpha }D_{x}^{\alpha }y+\frac{D_{x}^{\alpha }p(x)}{\lambda -p(x)}%
D_{x}^{\alpha }y+\left( \lambda -p(x)\right) ^{2}y=\left(
q(x)+p^{2}(x)\right) y+\frac{D_{x}^{\alpha }p(x)}{\lambda -p(x)}%
D_{x}^{\alpha }y.
\end{equation}

It is easily shown that the system of functions $\left\{ \cos \left( \frac{%
\lambda }{\alpha }x^{\alpha }-Q(x)\right) ,\text{ }\sin \left( \frac{\lambda 
}{\alpha }x^{\alpha }-Q(x)\right) \right\} $ is a fundamental system for the
following differential equation 
\begin{equation}
D_{x}^{\alpha }D_{x}^{\alpha }y+\frac{D_{x}^{\alpha }p(x)}{\lambda -p(x)}%
D_{x}^{\alpha }y+\left( \lambda -p(x)\right) ^{2}y=0.
\end{equation}

Thus, the solution of equation (1) satisfying the initial conditions (4)
provides the following integral equations%
\begin{eqnarray}
&&\left. \varphi (x,\lambda )=\cos \left( \tfrac{\lambda }{\alpha }x^{\alpha
}-Q(x)\right) +\tfrac{h}{\lambda -p(0)}\sin \left( \tfrac{\lambda }{\alpha }%
x^{\alpha }-Q(x)\right) \right. \medskip   \notag \\
&&\left. +\int_{0}^{x}\tfrac{\sin \left[ \frac{\lambda }{\alpha }\left(
x^{\alpha }-t^{\alpha }\right) -Q(x)+Q(t)\right] }{\lambda -p(t)}\left[
\left( q(t)+p^{2}(t)\right) \varphi (t,\lambda )+\tfrac{D_{t}^{\alpha }p(t)}{%
\lambda -p(t)}D_{t}^{\alpha }\varphi (t,\lambda )\right] d_{\alpha }t\right. 
\end{eqnarray}%
and by $\alpha -$differentiating (13) with respect to $x$%
\begin{eqnarray}
&&\left. D_{x}^{\alpha }\varphi (x,\lambda )=-\left( \lambda -p(x)\right)
\left\{ \sin \left( \tfrac{\lambda }{\alpha }x^{\alpha }-Q(x)\right) -\tfrac{%
h}{\lambda -p(0)}\cos \left( \tfrac{\lambda }{\alpha }x^{\alpha
}-Q(x)\right) \right. \right. \medskip   \notag \\
&&+\left. \int_{0}^{x}\tfrac{\cos \left[ \frac{\lambda }{\alpha }\left(
x^{\alpha }-t^{\alpha }\right) -Q(t)\right] }{\lambda -p(t)}\left[ \left(
q(t)+p^{2}(t)\right) \varphi (t,\lambda )+\tfrac{D_{t}^{\alpha }p(t)}{%
\lambda -p(t)}D_{t}^{\alpha }\varphi (t,\lambda )\right] d_{\alpha
}t\right\} .
\end{eqnarray}

For each fixed $\alpha ,$ denote%
\begin{equation*}
\mu _{1}\left( \lambda \right) :=\underset{0\leq x\leq \pi }{\max }%
\left\vert \varphi (x,\lambda )\exp \left( -\tfrac{\left\vert \tau
\right\vert }{\alpha }x^{\alpha }\right) \right\vert \text{ and }\mu
_{2}\left( \lambda \right) :=\underset{0\leq x\leq \pi }{\max }\left\vert
D_{x}^{\alpha }\varphi (x,\lambda )\exp \left( -\tfrac{\left\vert \tau
\right\vert }{\alpha }x^{\alpha }\right) \right\vert .
\end{equation*}

Since $\left\vert \cos \left( \tfrac{\lambda }{\alpha }x^{\alpha
}-Q(x)\right) \right\vert \leq \exp \left( \tfrac{\left\vert \tau
\right\vert }{\alpha }x^{\alpha }\right) $ and $\left\vert \sin \left( 
\tfrac{\lambda }{\alpha }x^{\alpha }-Q(x)\right) \right\vert \leq \exp
\left( \tfrac{\left\vert \tau \right\vert }{\alpha }x^{\alpha }\right) ,$
from (13) and (14) we find%
\begin{equation*}
\mu _{1}\left( \lambda \right) \leq C\left( 1+\tfrac{\left\vert h\right\vert 
}{\left\vert \lambda \right\vert }+\tfrac{\mu _{1}\left( \lambda \right) }{%
\left\vert \lambda \right\vert }+\tfrac{\mu _{2}\left( \lambda \right) }{%
\left\vert \lambda \right\vert ^{2}}\right) ,\text{ }\mu _{2}\left( \lambda
\right) \leq C\left( \left\vert \lambda \right\vert +\left\vert h\right\vert
+\mu _{1}\left( \lambda \right) +\tfrac{\mu _{2}\left( \lambda \right) }{%
\left\vert \lambda \right\vert }\right) .
\end{equation*}

Hence, we get 
\begin{equation*}
\mu _{1}\left( \lambda \right) \leq C,\mu _{2}\left( \lambda \right) \leq
C\left\vert \lambda \right\vert
\end{equation*}%
or%
\begin{equation*}
\varphi \left( x,\lambda \right) =O\left( \exp \left( \tfrac{\left\vert \tau
\right\vert }{\alpha }x^{\alpha }\right) \right) ,\text{ }D_{x}^{\alpha
}\varphi \left( x,\lambda \right) =O\left( \lambda \exp \left( \tfrac{%
\left\vert \tau \right\vert }{\alpha }x^{\alpha }\right) \right) ,\text{ }%
\left\vert \lambda \right\vert \rightarrow \infty .
\end{equation*}

Substituting these into (13) and (14) we obtain (8) and (9).
\end{proof}
\end{lemma}
\end{definition}

Applying successive approximations method to the equations (13), we can have
the more detalied asymptotic of the function $\varphi \left( x,\lambda
\right) $ as follows%
\begin{eqnarray}
&&\left. \varphi (x,\lambda )=\cos \left( \tfrac{\lambda }{\alpha }x^{\alpha
}-Q(x)\right) +\tfrac{p(x)-p(0)}{2\lambda }\cos \left( \tfrac{\lambda }{%
\alpha }x^{\alpha }-Q(x)\right) \right. \medskip  \notag \\
&&+\tfrac{1}{\lambda }\left( h+\tfrac{1}{2}\int_{0}^{x}\left(
q(t)+p^{2}(t)\right) d_{\alpha }t\right) \sin \left( \tfrac{\lambda }{\alpha 
}x^{\alpha }-Q(x)\right) \medskip  \notag \\
&&+\tfrac{1}{2\lambda }\int_{0}^{x}\left( q(t)+p^{2}(t)\right) \sin \left( 
\tfrac{\lambda }{\alpha }\left( x^{\alpha }-2t^{\alpha }\right)
-Q(x)+2Q(t)\right) d_{\alpha }t\medskip  \notag \\
&&-\tfrac{1}{2\lambda }\int_{0}^{x}D_{t}^{\alpha }p(t)\cos \left( \tfrac{%
\lambda }{\alpha }\left( x^{\alpha }-2t^{\alpha }\right) -Q(x)+2Q(t)\right)
d_{\alpha }t\medskip  \notag \\
&&+\tfrac{1}{\lambda ^{2}}\left[ \tfrac{h\left( p(x)+p(0)\right) }{2}\right.
\medskip \\
&&\left. +\tfrac{1}{4}\int_{0}^{x}\left( q(t)+p^{2}(t)\right) \left(
p(x)-p(0)+2p(t)\right) d_{\alpha }t\right] \sin \left( \tfrac{\lambda }{%
\alpha }x^{\alpha }-Q(x)\right) \medskip  \notag \\
&&+\tfrac{1}{\lambda ^{2}}\left[ \tfrac{p^{1+\alpha }(x)-p^{1+\alpha }(0)}{2}%
+\tfrac{\left( p(x)-p(0)\right) ^{1+\alpha }}{4\left( 1+\alpha \right) }-%
\tfrac{h}{2}\int_{0}^{x}\left( q(t)+p^{2}(t)\right) d_{\alpha }t\right.
\medskip  \notag \\
&&\left. -\tfrac{1}{8}\left( \int_{0}^{x}\left( q(t)+p^{2}(t)\right)
d_{\alpha }t\right) ^{2}\right] \cos \left( \tfrac{\lambda }{\alpha }%
x^{\alpha }-Q(x)\right) \medskip  \notag \\
&&+O\left( \tfrac{1}{\left\vert \lambda \right\vert ^{3}}\exp \left( \tfrac{%
\left\vert \tau \right\vert }{\alpha }x^{\alpha }\right) \right) ,\text{ }%
\left\vert \lambda \right\vert \rightarrow \infty ,  \notag
\end{eqnarray}
uniformly with respect to $x\in \lbrack 0,\pi ]$ for each fixed $\alpha .$

The eigenvalues of $L_{\alpha }$ coincide with the zeros of its
characteristic function $\Delta (\lambda )=V\left( \varphi \right)
=D_{x}^{\alpha }\varphi (\pi ,\lambda )+H\varphi (\pi ,\lambda ).$ Thus,
using the formula (15) we can establish the following asymptotic%
\begin{eqnarray}
&&\left. \Delta (\lambda )=-\lambda \sin \left( \tfrac{\lambda }{\alpha }\pi
^{\alpha }-c_{0}\right) +\tfrac{\left( p(\pi )+p(0)\right) }{2}\sin \left( 
\tfrac{\lambda }{\alpha }\pi ^{\alpha }-c_{0}\right) \right. \medskip  \notag
\\
&&+c_{1}\cos \left( \tfrac{\lambda }{\alpha }\pi ^{\alpha }-c_{0}\right) +%
\tfrac{c_{2}}{\lambda }\sin \left( \tfrac{\lambda }{\alpha }\pi ^{\alpha
}-c_{0}\right) +\tfrac{c_{3}}{\lambda }\cos \left( \tfrac{\lambda }{\alpha }%
\pi ^{\alpha }-c_{0}\right) \medskip  \notag \\
&&+\tfrac{1}{2}\int_{0}^{\pi }\left( q(t)+p^{2}(t)\right) \sin \left[ \tfrac{%
\lambda }{\alpha }\left( \pi ^{\alpha }-2t^{\alpha }\right) -Q(\pi )+2Q(t)%
\right] d_{\alpha }t\medskip \\
&&+\tfrac{1}{2}\int_{0}^{\pi }D_{t}^{\alpha }p(t)\cos \left[ \tfrac{\lambda 
}{\alpha }\left( \pi ^{\alpha }-2t^{\alpha }\right) -Q(\pi )+2Q(t)\right]
d_{\alpha }t\medskip  \notag \\
&&+O\left( \tfrac{1}{\left\vert \lambda \right\vert ^{2}}\exp \left( \tfrac{%
\left\vert \tau \right\vert }{\alpha }\pi ^{\alpha }\right) \right) ,\text{ }%
\left\vert \lambda \right\vert \rightarrow \infty ,  \notag
\end{eqnarray}%
where, for each fixed $\alpha $\medskip

$c_{0}=Q(\pi )=\int_{0}^{\pi }p(t)d_{\alpha }t,$\medskip

$c_{1}=h+H+\tfrac{1}{2}\int_{0}^{\pi }\left( q(t)+p^{2}(t)\right) d_{\alpha
}t,$\medskip

$c_{2}=\tfrac{p(\pi )\left( p(\pi )-p(0)\right) }{2}-\tfrac{p^{1+\alpha
}(\pi )-p^{1+\alpha }(0)}{2}-\tfrac{\left( p(\pi )-p(0)\right) ^{1+\alpha }}{%
4\left( 1+\alpha \right) }+hH$\medskip

$+\tfrac{h+H}{2}\int_{0}^{\pi }\left( q(t)+p^{2}(t)\right) d_{\alpha }t+%
\tfrac{1}{8}\left( \int_{0}^{\pi }\left( q(t)+p^{2}(t)\right) d_{\alpha
}t\right) ^{2},$\medskip

$c_{3}=\tfrac{\left( H-h\right) \left( p(\pi )-p(0)\right) }{2}+\tfrac{1}{4}%
\int_{0}^{\pi }\left( q(t)+p^{2}(t)\right) \left( 2p(t)-p(\pi )-p(0)\right)
d_{\alpha }t.$\medskip

Take a circle $\Gamma _{N}=\left\{ \left. \lambda \right\vert \text{ }%
\left\vert \lambda \right\vert =\tfrac{\alpha }{\pi ^{\alpha -1}}\left( N+%
\tfrac{1}{2}\right) ,\text{ }N=0,1,2,\ldots \right\} $ in the $\lambda -$%
plane. By the standard method using (16) and Rouche's theorem (see [50]) and
taking $\Delta (\lambda _{n})=0$\ one can prove that in the circle $\Gamma
_{N}$, there exist exactly $\left\vert n\right\vert $ eigenvalues $\lambda
_{n}$ and have the form%
\begin{equation}
\left. \lambda _{n}=\frac{n\alpha }{\pi ^{\alpha -1}}+\frac{\alpha c_{0}}{%
\pi ^{\alpha }}+\frac{c_{1}+A_{n}}{n\pi }+O\left( \dfrac{1}{n^{2}}\right)
\right. ,\text{ }\left\vert n\right\vert \rightarrow \infty ,
\end{equation}%
where $n\in 
\mathbb{Z}
$ and for each fixed $\alpha $%
\begin{eqnarray*}
&&\left. A_{n}=\tfrac{1}{2}\int_{0}^{\pi }\left( q(t)+p^{2}(t)\right) \cos
\left( \tfrac{2nt^{\alpha }}{\pi ^{\alpha -1}}+\tfrac{2c_{0}t^{\alpha }}{\pi
^{\alpha }}-2Q(t)\right) d_{\alpha }t\right. \medskip \\
&&-\left. \tfrac{1}{2}\int_{0}^{\pi }D_{t}^{\alpha }p(t)\sin \left( \tfrac{%
2nt^{\alpha }}{\pi ^{\alpha -1}}+\tfrac{2c_{0}t^{\alpha }}{\pi ^{\alpha }}%
-2Q(t)\right) d_{\alpha }t.\right.
\end{eqnarray*}

\begin{corollary}
According to (17) for each fixed $\alpha $ and sufficiently large $%
\left\vert n\right\vert $ the eigenvalues $\lambda _{n}$ are real and simple.
\end{corollary}

\section{\textbf{Main Results}}

In this section, using the contour integration method, we will find formulas
which are so-called regularized trace formulas.

\begin{theorem}
Let $\left\{ \lambda _{n}\right\} _{n\geq 0}$ be the sequence of the
eigenvalues of the problem $L_{\alpha }$. Then, for each fixed $\alpha ,$
the following trace formulas are valid:%
\begin{eqnarray}
&&\left. 2\left( \lambda _{0}-c_{0}\right) +\sum_{n=1}^{\infty }\left[
\lambda _{n}+\lambda _{-n}-2c_{0}-\tfrac{1}{\alpha \pi ^{1-\alpha }}\tfrac{%
B_{n}}{n}\right] =\tfrac{p(\pi )+p(0)-2c_{0}}{2}\right. \medskip  \notag \\
&&-\tfrac{1}{2}\int_{0}^{\pi }\left( 1-\tfrac{2t^{\alpha }}{\alpha }\right)
\left( q(t)+p^{2}(t)\right) \sin 2Q(t)d_{\alpha }t\medskip \\
&&+\tfrac{1}{2}\int_{0}^{\pi }\left( 1-\tfrac{2t^{\alpha }}{\alpha }\right)
D_{t}^{\alpha }p(t)\cos 2Q(t)d_{\alpha }t  \notag
\end{eqnarray}%
and%
\begin{eqnarray}
&&\left. 2\left( \lambda _{0}-c_{0}\right) ^{2}+\sum_{n=1}^{\infty }\left[
\left( \lambda _{n}-c_{0}\right) ^{2}+\left( \lambda _{-n}-c_{0}\right)
^{2}-2\left( \tfrac{n\alpha }{\pi ^{1-\alpha }}\right) ^{2}-\tfrac{4\alpha
c_{1}}{\pi ^{\alpha }}-\tfrac{2\alpha }{\pi ^{\alpha }}C_{n}\right] \right.
\medskip  \notag \\
&&\left. =\tfrac{2\alpha }{\pi ^{\alpha }}\left( h+H+\tfrac{1}{2}%
\int_{0}^{\pi }\left( q(t)+p^{2}(t)\right) d_{\alpha }t\right) \right.
\medskip  \notag \\
&&+\tfrac{\alpha }{\pi ^{\alpha }}\int_{0}^{\pi }\left( q(t)+p^{2}(t)\right)
\cos 2Q(t)d_{\alpha }t\medskip  \notag \\
&&+\tfrac{\alpha }{\pi ^{\alpha }}\int_{0}^{\pi }D_{t}^{\alpha }p(t)\sin
2Q(t)d_{\alpha }t\medskip \\
&&+\left( p(\pi )-c_{0}\right) \left( p(\pi )-p(0)\right) -\left( p(\pi
)-c_{0}\right) ^{1+\alpha }+\left( p(0)-c_{0}\right) ^{1+\alpha }\medskip 
\notag \\
&&-\tfrac{\left( p(\pi )-p(0)\right) ^{1+\alpha }}{2\left( 1+\alpha \right) }%
+2hH+\left( h+H\right) \int_{0}^{\pi }\left( q(t)+p^{2}(t)\right) d_{\alpha
}t\medskip  \notag \\
&&+\tfrac{1}{4}\left( \int_{0}^{\pi }\left( q(t)+p^{2}(t)\right) d_{\alpha
}t\right) ^{2}  \notag
\end{eqnarray}%
where$\medskip $\newline
$B_{n}=\int_{0}^{\pi }\left( q(t)+p^{2}(t)\right) \sin \left( \frac{2n\alpha
t^{\alpha }}{\pi ^{2\alpha -1}}\right) \sin 2Q(t)d_{\alpha }t\medskip $%
\newline
$-\int_{0}^{\pi }D_{t}^{\alpha }p(t)\sin \left( \frac{2n\alpha t^{\alpha }}{%
\pi ^{2\alpha -1}}\right) \cos 2Q(t)d_{\alpha }t,\medskip $\newline
$C_{n}=\int_{0}^{\pi }\left( q(t)+p^{2}(t)\right) \cos \left( \frac{2n\alpha
t^{\alpha }}{\pi ^{2\alpha -1}}\right) \cos 2Q(t)d_{\alpha }t\medskip $%
\newline
$-\int_{0}^{\pi }D_{t}^{\alpha }p(t)\cos \left( \frac{2n\alpha t^{\alpha }}{%
\pi ^{2\alpha -1}}\right) \sin 2Q(t)d_{\alpha }t.\medskip $
\end{theorem}

\begin{proof}
Firstly, we consider the case $c_{0}=0.$

Denote%
\begin{equation}
\Delta _{0}(\lambda )=-\lambda \sin \left( \tfrac{\lambda }{\alpha }\pi
^{\alpha }\right) .
\end{equation}

It is clear that the zeros of the function $\Delta _{0}(\lambda )$ is%
\begin{equation*}
\mu _{n}=\frac{n\alpha }{\pi ^{\alpha -1}},\text{ }n\in 
\mathbb{Z}
,
\end{equation*}%
where only $\mu _{0}=0$ is double. We note that for each fixed $\alpha $ and
sufficiently large $N,$ the eigenvalues $\lambda _{n}$ which are the zeros
of $\Delta (\lambda )$ are inside $\Gamma _{N}$ and the numbers $\mu _{n}$
do not lie on the contour $\Gamma _{N}.$

Let $\Delta (\lambda )=\lambda \left( \lambda -\lambda _{n}\right) $ and $%
\Delta _{0}(\lambda )=\lambda \left( \lambda -\frac{n\alpha }{\pi ^{\alpha
-1}}\right) ,$ then, from the logaritmic derivatives of the functions $%
\Delta (\lambda )$ and $\Delta _{0}(\lambda ),$ we have%
\begin{equation*}
\lambda \frac{\overset{\cdot }{\Delta }(\lambda )}{\Delta (\lambda )}=\frac{%
2\lambda -\lambda _{n}}{\lambda -\lambda _{n}}\text{ and }\lambda \frac{%
\overset{\cdot }{\Delta }_{0}(\lambda )}{\Delta _{0}(\lambda )}=\frac{%
2\lambda -\frac{n\alpha }{\pi ^{\alpha -1}}}{\lambda -\frac{n\alpha }{\pi
^{\alpha -1}}}
\end{equation*}%
respectively, where $\overset{\cdot }{\Delta }=\frac{d}{d\lambda }.$

Thus, from the residue theorem, the following equalities are valid: 
\begin{equation}
\tfrac{1}{2\pi i}\oint_{\Gamma _{N}}\lambda \tfrac{\overset{\cdot }{\Delta }%
(\lambda )}{\Delta (\lambda )}d\lambda =\sum_{n=-N}^{N}\func{Re}\text{s}%
\left( \lambda \tfrac{\overset{\cdot }{\Delta }(\lambda )}{\Delta (\lambda )}%
,\lambda _{n}\right) =\sum_{n=0}^{N}\left( \lambda _{n}+\lambda _{-n}\right)
\end{equation}%
and similarly%
\begin{equation}
\tfrac{1}{2\pi i}\oint_{\Gamma _{N}}\lambda \tfrac{\overset{\cdot }{\Delta }%
_{0}(\lambda )}{\Delta _{0}(\lambda )}d\lambda =\sum_{n=-N}^{N}\func{Re}%
\text{s}\left( \lambda \tfrac{\overset{\cdot }{\Delta }_{0}(\lambda )}{%
\Delta _{0}(\lambda )},\tfrac{n\alpha }{\pi ^{\alpha -1}}\right)
=\sum_{n=0}^{N}\left[ \tfrac{n\alpha }{\pi ^{\alpha -1}}+\left( -\tfrac{%
n\alpha }{\pi ^{\alpha -1}}\right) \right] .
\end{equation}

Subtracting (21) and (22) side by side, we get 
\begin{eqnarray}
&&\left. \sum_{n=0}^{N}\left( \lambda _{n}+\lambda _{-n}\right) =\tfrac{1}{%
2\pi i}\oint_{\Gamma _{N}}\lambda \left( \tfrac{\overset{\cdot }{\Delta }%
(\lambda )}{\Delta (\lambda )}-\tfrac{\overset{\cdot }{\Delta }_{0}(\lambda )%
}{\Delta _{0}(\lambda )}\right) d\lambda \right. \medskip  \notag \\
&&\left. =\tfrac{1}{2\pi i}\oint_{\Gamma _{N}}\lambda d\left( \ln \tfrac{%
\Delta (\lambda )}{\Delta _{0}(\lambda )}\right) =-\tfrac{1}{2\pi i}%
\oint_{\Gamma _{N}}\ln \tfrac{\Delta (\lambda )}{\Delta _{0}(\lambda )}%
d\lambda \right. .
\end{eqnarray}

On the other hand, it follows from (16) and (20) that%
\begin{eqnarray}
&&\left. \tfrac{\Delta (\lambda )}{\Delta _{0}(\lambda )}=1-\tfrac{p(\pi
)+p(0)+2A\left( \lambda \right) }{2\lambda }-\tfrac{c_{1}+B\left( \lambda
\right) }{\lambda }\cot \left( \tfrac{\lambda }{\alpha }\pi ^{\alpha
}\right) \right. \medskip  \notag \\
&&\left. -\tfrac{c_{2}}{\lambda ^{2}}-\tfrac{c_{3}}{\lambda ^{2}}\cot \left( 
\tfrac{\lambda }{\alpha }\pi ^{\alpha }\right) +O\left( \tfrac{1}{\lambda
^{3}}\exp \left( \tfrac{\left\vert \tau \right\vert }{\alpha }\pi ^{\alpha
}\right) \right) ,\text{ on }\Gamma _{N},\right.
\end{eqnarray}%
where$\medskip $\newline
$A\left( \lambda \right) =\tfrac{1}{2}\int_{0}^{\pi }\left(
q(t)+p^{2}(t)\right) \sin \left( \tfrac{2\lambda t^{\alpha }}{\alpha }%
-2Q(t)\right) d_{\alpha }t+\tfrac{1}{2}\int_{0}^{\pi }D_{t}^{\alpha
}p(t)\cos \left( \tfrac{2\lambda t^{\alpha }}{\alpha }-2Q(t)\right)
d_{\alpha }t,\medskip $\newline
$B\left( \lambda \right) =\tfrac{1}{2}\int_{0}^{\pi }\left(
q(t)+p^{2}(t)\right) \cos \left( \tfrac{2\lambda t^{\alpha }}{\alpha }%
-2Q(t)\right) d_{\alpha }t-\tfrac{1}{2}\int_{0}^{\pi }D_{t}^{\alpha
}p(t)\sin \left( \tfrac{2\lambda t^{\alpha }}{\alpha }-2Q(t)\right)
d_{\alpha }t.\medskip $

Taking Taylor's expansion formula for the ln$\left( 1-u\right) $ into
account, we get 
\begin{equation}
\ln \left( \tfrac{\Delta (\lambda )}{\Delta _{0}(\lambda )}\right) =-\tfrac{%
p(\pi )+p(0)+2A\left( \lambda \right) }{2\lambda }-\tfrac{c_{1}+B\left(
\lambda \right) }{\lambda }\cot \left( \tfrac{\lambda }{\alpha }\pi ^{\alpha
}\right) +O\left( \tfrac{1}{\lambda ^{2}}\right) ,\text{ on }\Gamma _{N},
\end{equation}

Thus, substituting (24) into (23) we find%
\begin{eqnarray}
&&\left. \sum_{n=0}^{N}\left( \lambda _{n}+\lambda _{-n}\right) =\tfrac{1}{%
2\pi i}\oint_{\Gamma _{N}}\tfrac{p(\pi )+p(0)+2A\left( \lambda \right) }{%
2\lambda }d\lambda \right. \medskip  \notag \\
&&\left. +\tfrac{1}{2\pi i}\oint_{\Gamma _{N}}\tfrac{c_{1}+B\left( \lambda
\right) }{\lambda }\cot \left( \tfrac{\lambda }{\alpha }\pi ^{\alpha
}\right) d\lambda +\tfrac{1}{2\pi i}\oint_{\Gamma _{N}}O\left( \tfrac{1}{%
\lambda ^{2}}\right) d\lambda \right. .
\end{eqnarray}

By the well-known formulas such as the generalized Cauchy integral formula,
the residue theorem and $\cot z=\tfrac{1}{z}+2z\sum_{n=1}^{\infty }\frac{1}{%
z^{2}-n^{2}\pi ^{2}}$ (see [51] for more details), the contour integrals in
(26) calculate that%
\begin{equation}
\tfrac{1}{2\pi i}\oint_{\Gamma _{N}}\tfrac{p(\pi )+p(0)+2A\left( \lambda
\right) }{2\lambda }d\lambda =\tfrac{p(\pi )+p(0)}{2}+A\left( 0\right) ,
\end{equation}%
\begin{eqnarray}
&&\left. \tfrac{1}{2\pi i}\oint_{\Gamma _{N}}\tfrac{c_{1}+B\left( \lambda
\right) }{\lambda }\cot \left( \tfrac{\lambda }{\alpha }\pi ^{\alpha
}\right) d\lambda =\tfrac{\alpha }{\pi ^{\alpha }}\overset{\cdot }{B}%
(0)\right. \medskip  \notag \\
&&\left. +\tfrac{1}{\alpha \pi ^{1-\alpha }}\sum_{n=1}^{N}\tfrac{1}{n}\left(
B\left( \tfrac{n\alpha ^{2}}{\pi ^{2\alpha -1}}\right) -B\left( -\tfrac{%
n\alpha ^{2}}{\pi ^{2\alpha -1}}\right) \right) ,\right.
\end{eqnarray}%
and for sufficiently large $N$ and each fixed $\alpha $ 
\begin{equation}
\left\vert \oint_{\Gamma _{N}}O\left( \tfrac{1}{\lambda ^{2}}\right)
d\lambda \right\vert =O\left( \tfrac{1}{N}\right) .
\end{equation}

From (26)-(29), we get%
\begin{eqnarray}
&&\left. 2\lambda _{0}+\sum_{n=1}^{N}\left[ \lambda _{n}+\lambda _{-n}-%
\tfrac{1}{n\alpha \pi ^{1-\alpha }}\left( B\left( \tfrac{n\alpha ^{2}}{\pi
^{2\alpha -1}}\right) -B\left( -\tfrac{n\alpha ^{2}}{\pi ^{2\alpha -1}}%
\right) \right) \right] \right. \medskip  \notag \\
&&\left. =\tfrac{p(\pi )+p(0)}{2}+A\left( 0\right) +\tfrac{\alpha }{\pi
^{\alpha }}\overset{\cdot }{B}(0)+O\left( \tfrac{1}{N}\right) \right. .
\end{eqnarray}

For $N\rightarrow \infty $ in (30), 
\begin{eqnarray}
&&\left. 2\lambda _{0}+\sum_{n=1}^{\infty }\left[ \lambda _{n}+\lambda _{-n}-%
\tfrac{1}{n\alpha \pi ^{1-\alpha }}\left( B\left( \tfrac{n\alpha ^{2}}{\pi
^{2\alpha -1}}\right) -B\left( -\tfrac{n\alpha ^{2}}{\pi ^{2\alpha -1}}%
\right) \right) \right] \right. \medskip  \notag \\
&&\left. =\tfrac{p(\pi )+p(0)}{2}+A\left( 0\right) +\tfrac{\alpha }{\pi
^{\alpha }}\overset{\cdot }{B}(0)\right.
\end{eqnarray}%
is obtained.

Now we consider the case $c_{0}\neq 0.$

It is obvious that we can rewrite the equation (1) as%
\begin{equation*}
-D_{x}^{\alpha }D_{x}^{\alpha }y+\left[ 2\left( \lambda -c_{0}\right) \left(
p(x)-c_{0}\right) +q(x)+2c_{0}p(x)-c_{0}^{2}\right] y=\left( \lambda
-c_{0}\right) ^{2}y.
\end{equation*}

Denote $\lambda -c_{0}=\widetilde{\lambda },$ $q(x)+2c_{0}p(x)-c_{0}^{2}=%
\widetilde{q}(x)$ and $p(x)-c_{0}=\widetilde{p}(x).$ Thus,%
\begin{equation}
-D_{x}^{\alpha }D_{x}^{\alpha }y+\left[ 2\widetilde{\lambda }\widetilde{p}%
(x)+\widetilde{q}(x)\right] y=\widetilde{\lambda }^{2}y.
\end{equation}

For the equation (32), at case $\widetilde{c}_{0}=\int_{0}^{\pi }\widetilde{p%
}(t)d_{\alpha }t=0$, according to (31) we obtain%
\begin{eqnarray}
&&\left. 2\widetilde{\lambda }_{0}+\sum_{n=1}^{\infty }\left[ \widetilde{%
\lambda }_{n}+\widetilde{\lambda }_{-n}-\tfrac{B_{n}}{n\alpha \pi ^{1-\alpha
}}\right] \right. \medskip  \notag \\
&&\left. =\tfrac{\widetilde{p}(\pi )+\widetilde{p}(0)}{2}+\widetilde{A}%
\left( 0\right) +\tfrac{\alpha }{\pi ^{\alpha }}\overset{\cdot }{\widetilde{B%
}}(0)\right. ,
\end{eqnarray}%
where, 
\begin{equation*}
A\left( \widetilde{\lambda }\right) :=\widetilde{A}\left( \lambda \right)
=A\left( \lambda \right) ,B\left( \widetilde{\lambda }\right) :=\widetilde{B}%
\left( \lambda \right) =B\left( \lambda \right) \text{ and }\widetilde{B}%
\left( \tfrac{n\alpha ^{2}}{\pi ^{2\alpha -1}}\right) -\widetilde{B}\left( -%
\tfrac{n\alpha ^{2}}{\pi ^{2\alpha -1}}\right) =B_{n}.
\end{equation*}

Substituting the expressions of $\widetilde{\lambda }_{n},$ $\widetilde{q}%
(x) $ and $\widetilde{p}(x)$ into (33), we arrive at (18).

Similarly, we prove that the formula (19) is true.

We consider the case $c_{0}=0$ again$.$ Denote $\Delta (\lambda )=\lambda
^{2}\left( \lambda -\lambda _{n}\right) $ and $\Delta _{0}(\lambda )=\lambda
^{2}\left( \lambda -\frac{n\alpha }{\pi ^{\alpha -1}}\right) .$ Then, 
\begin{equation*}
\lambda ^{2}\frac{\overset{\cdot }{\Delta }(\lambda )}{\Delta (\lambda )}=%
\frac{3\lambda ^{2}-2\lambda \lambda _{n}}{\lambda -\lambda _{n}}\text{ and }%
\lambda ^{2}\frac{\overset{\cdot }{\Delta }_{0}(\lambda )}{\Delta
_{0}(\lambda )}=\frac{3\lambda ^{2}-2\lambda \frac{n\alpha }{\pi ^{\alpha -1}%
}}{\lambda -\frac{n\alpha }{\pi ^{\alpha -1}}}.
\end{equation*}

Thus, the following equalities are valid: 
\begin{equation}
\tfrac{1}{2\pi i}\oint_{\Gamma _{N}}\lambda ^{2}\tfrac{\overset{\cdot }{%
\Delta }(\lambda )}{\Delta (\lambda )}d\lambda =\sum_{n=-N}^{N}\func{Re}%
\text{s}\left( \lambda ^{2}\tfrac{\overset{\cdot }{\Delta }(\lambda )}{%
\Delta (\lambda )},\lambda _{n}\right) =\sum_{n=0}^{N}\left( \lambda
_{n}^{2}+\lambda _{-n}^{2}\right)
\end{equation}%
and 
\begin{equation}
\tfrac{1}{2\pi i}\oint_{\Gamma _{N}}\lambda ^{2}\tfrac{\overset{\cdot }{%
\Delta }_{0}(\lambda )}{\Delta _{0}(\lambda )}d\lambda =\sum_{n=-N}^{N}\func{%
Re}\text{s}\left( \lambda \tfrac{\overset{\cdot }{\Delta }_{0}(\lambda )}{%
\Delta _{0}(\lambda )},\tfrac{n\alpha }{\pi ^{\alpha -1}}\right)
=\sum_{n=0}^{N}\left[ \left( \tfrac{n\alpha }{\pi ^{\alpha -1}}\right)
^{2}+\left( -\tfrac{n\alpha }{\pi ^{\alpha -1}}\right) ^{2}\right] .
\end{equation}%
{\small \ }

Subtracting (34) and (35) side by side, we get 
\begin{eqnarray}
&&\left. \sum_{n=0}^{N}\left[ \lambda _{n}^{2}+\lambda _{-n}^{2}-\left( 
\tfrac{n\alpha }{\pi ^{\alpha -1}}\right) ^{2}-\left( -\tfrac{n\alpha }{\pi
^{\alpha -1}}\right) ^{2}\right] =\tfrac{1}{2\pi i}\oint_{\Gamma
_{N}}\lambda ^{2}\left( \tfrac{\overset{\cdot }{\Delta }(\lambda )}{\Delta
(\lambda )}-\tfrac{\overset{\cdot }{\Delta }_{0}(\lambda )}{\Delta
_{0}(\lambda )}\right) d\lambda \right. \medskip  \notag \\
&&\left. =\tfrac{1}{2\pi i}\oint_{\Gamma _{N}}\lambda ^{2}d\left( \ln \tfrac{%
\Delta (\lambda )}{\Delta _{0}(\lambda )}\right) =-\tfrac{1}{2\pi i}%
\oint_{\Gamma _{N}}2\lambda \ln \tfrac{\Delta (\lambda )}{\Delta
_{0}(\lambda )}d\lambda .\right.
\end{eqnarray}

From (24), we obtain 
\begin{eqnarray}
&&\left. \ln \left( \tfrac{\Delta (\lambda )}{\Delta _{0}(\lambda )}\right)
=-\tfrac{p(\pi )+p(0)+2A\left( \lambda \right) }{2\lambda }-\tfrac{%
c_{1}+B\left( \lambda \right) }{\lambda }\cot \left( \tfrac{\lambda }{\alpha 
}\pi ^{\alpha }\right) \right. \medskip  \notag \\
&&\left. -\tfrac{c_{2}}{\lambda ^{2}}-\tfrac{c_{3}}{\lambda ^{2}}\cot \left( 
\tfrac{\lambda }{\alpha }\pi ^{\alpha }\right) +O\left( \tfrac{1}{\lambda
^{3}}\right) ,\text{ on }\Gamma _{N}\right. .
\end{eqnarray}

Thus, the integral on the right side of (36) is written as%
\begin{eqnarray}
&&\left. -\tfrac{1}{2\pi i}\oint_{\Gamma _{N}}2\lambda \ln \tfrac{\Delta
(\lambda )}{\Delta _{0}(\lambda )}d\lambda =\tfrac{1}{2\pi i}\oint_{\Gamma
_{N}}\left( p(\pi )+p(0)+2A\left( \lambda \right) \right) d\lambda \right.
\medskip  \notag \\
&&\left. +\tfrac{1}{2\pi i}\oint_{\Gamma _{N}}2\left( c_{1}+B\left( \lambda
\right) \right) \cot \left( \tfrac{\lambda }{\alpha }\pi ^{\alpha }\right)
d\lambda +\tfrac{1}{2\pi i}\oint_{\Gamma _{N}}\tfrac{2c_{2}}{\lambda }%
d\lambda \right. \medskip \\
&&\left. +\tfrac{1}{2\pi i}\oint_{\Gamma _{N}}\tfrac{2c_{3}}{\lambda }\cot
\left( \tfrac{\lambda }{\alpha }\pi ^{\alpha }\right) d\lambda +\tfrac{1}{%
2\pi i}\oint_{\Gamma _{N}}O\left( \tfrac{1}{\lambda ^{2}}\right) d\lambda
\right. .  \notag
\end{eqnarray}

For the contour integrals in (38), we have%
\begin{equation}
\tfrac{1}{2\pi i}\oint_{\Gamma _{N}}\left[ p(\pi )+p(0)+2A\left( \lambda
\right) \right] d\lambda =0,
\end{equation}%
\begin{eqnarray}
&&\left. \tfrac{1}{2\pi i}\oint_{\Gamma _{N}}2\left[ c_{1}+B\left( \lambda
\right) \right] \cot \left( \tfrac{\lambda }{\alpha }\pi ^{\alpha }\right)
d\lambda =\tfrac{2\alpha c_{1}}{\pi ^{\alpha }}+\tfrac{4\alpha c_{1}}{\pi
^{\alpha }}N\right. \medskip  \notag \\
&&\left. +\tfrac{2\alpha }{\pi ^{\alpha }}B(0)+\tfrac{2\alpha }{\pi ^{\alpha
}}\sum_{n=1}^{N}\left[ B\left( \tfrac{n\alpha ^{2}}{\pi ^{2\alpha -1}}%
\right) +B\left( -\tfrac{n\alpha ^{2}}{\pi ^{2\alpha -1}}\right) \right]
\right. ,
\end{eqnarray}%
\begin{equation}
\tfrac{1}{2\pi i}\oint_{\Gamma _{N}}\tfrac{2c_{2}}{\lambda }d\lambda =2c_{2},
\end{equation}%
\begin{equation}
\tfrac{1}{2\pi i}\oint_{\Gamma _{N}}\tfrac{2c_{3}}{\lambda }\cot \left( 
\tfrac{\lambda }{\alpha }\pi ^{\alpha }\right) d\lambda =0.
\end{equation}

Substituting the expressions of (29) and (39)-(42) into (38), we arrive%
\begin{eqnarray*}
&&\left. -\tfrac{1}{2\pi i}\oint_{\Gamma _{N}}2\lambda \ln \tfrac{\Delta
(\lambda )}{\Delta _{0}(\lambda )}d\lambda =\tfrac{2\alpha c_{1}}{\pi
^{\alpha }}+\tfrac{4\alpha c_{1}}{\pi ^{\alpha }}N\right. \medskip \\
&&\left. +\tfrac{2\alpha }{\pi ^{\alpha }}B(0)+\tfrac{2\alpha }{\pi ^{\alpha
}}\sum_{n=1}^{N}\left[ B\left( \tfrac{n\alpha ^{2}}{\pi ^{2\alpha -1}}%
\right) +B\left( -\tfrac{n\alpha ^{2}}{\pi ^{2\alpha -1}}\right) \right]
+2c_{2}+O\left( \tfrac{1}{N}\right) \right.
\end{eqnarray*}%
and from (36)%
\begin{eqnarray}
&&\left. 2\lambda _{0}^{2}+\sum_{n=0}^{N}\left[ \lambda _{n}^{2}+\lambda
_{-n}^{2}-2\left( \tfrac{n\alpha }{\pi ^{\alpha -1}}\right) ^{2}-\tfrac{%
4\alpha c_{1}}{\pi ^{\alpha }}-\tfrac{2\alpha }{\pi ^{\alpha }}\left(
B\left( \tfrac{n\alpha ^{2}}{\pi ^{2\alpha -1}}\right) +B\left( -\tfrac{%
n\alpha ^{2}}{\pi ^{2\alpha -1}}\right) \right) \right] \right. \medskip 
\notag \\
&&\left. =\tfrac{2\alpha c_{1}}{\pi ^{\alpha }}+\tfrac{2\alpha }{\pi
^{\alpha }}B(0)+2c_{2}+O\left( \tfrac{1}{N}\right) \right. .
\end{eqnarray}

For $N\rightarrow \infty $ in (43)%
\begin{eqnarray}
&&\left. 2\lambda _{0}^{2}+\sum_{n=0}^{\infty }\left[ \lambda
_{n}^{2}+\lambda _{-n}^{2}-2\left( \tfrac{n\alpha }{\pi ^{\alpha -1}}\right)
^{2}-\tfrac{4\alpha c_{1}}{\pi ^{\alpha }}-\tfrac{2\alpha }{\pi ^{\alpha }}%
\left( B\left( \tfrac{n\alpha ^{2}}{\pi ^{2\alpha -1}}\right) +B\left( -%
\tfrac{n\alpha ^{2}}{\pi ^{2\alpha -1}}\right) \right) \right] \right.
\medskip  \notag \\
&&\left. =\tfrac{2\alpha c_{1}}{\pi ^{\alpha }}+\tfrac{2\alpha }{\pi
^{\alpha }}B(0)+2c_{2}\right.
\end{eqnarray}%
is obtained.

In the case $c_{0}\neq 0,$ from (44) we can write that%
\begin{eqnarray}
&&\left. 2\widetilde{\lambda }_{0}^{2}+\sum_{n=0}^{\infty }\left[ \widetilde{%
\lambda }_{n}^{2}+\widetilde{\lambda }_{-n}^{2}-2\left( \tfrac{n\alpha }{\pi
^{\alpha -1}}\right) ^{2}-\tfrac{4\alpha \widetilde{c}_{1}}{\pi ^{\alpha }}-%
\tfrac{2\alpha }{\pi ^{\alpha }}C_{n}\right] \right. \medskip  \notag \\
&&\left. =\tfrac{2\alpha \widetilde{c}_{1}}{\pi ^{\alpha }}+\tfrac{2\alpha }{%
\pi ^{\alpha }}\widetilde{B}(0)+2\widetilde{c}_{2}\right. .
\end{eqnarray}

Substituting the known expressions of $\widetilde{\lambda }_{n},$ $%
\widetilde{q}(x)$ and $\widetilde{p}(x)$ into (45), we arrive the formula
(19), where$\medskip $

$\widetilde{B}\left( \tfrac{n\alpha ^{2}}{\pi ^{2\alpha -1}}\right) +%
\widetilde{B}\left( -\tfrac{n\alpha ^{2}}{\pi ^{2\alpha -1}}\right)
=C_{n},\medskip $

$\widetilde{c}_{1}=c_{1},\medskip $

$\widetilde{c}_{2}=\frac{\left( p(\pi )-c_{0}\right) \left( p(\pi
)-p(0)\right) }{2}-\frac{\left( p(\pi )-c_{0}\right) ^{1+\alpha }+\left(
p(0)-c_{0}\right) ^{1+\alpha }}{2}-\tfrac{\left( p(\pi )-p(0)\right)
^{1+\alpha }}{4\left( 1+\alpha \right) }+hH\medskip $

$+\frac{\left( h+H\right) }{2}\int_{0}^{\pi }\left( q(t)+p^{2}(t)\right)
d_{\alpha }t+\tfrac{1}{8}\left( \int_{0}^{\pi }\left( q(t)+p^{2}(t)\right)
d_{\alpha }t\right) ^{2}.\medskip $
\end{proof}

\end{document}